\documentclass[11pt]{amsart}

\usepackage[english]{babel}
\usepackage{amsmath, amsfonts, amssymb, amsthm, epsfig, graphicx, url, hyperref, color, tikz, cancel}

\newtheorem{theorem}{Theorem}[section]
\newtheorem{lemma}{Lemma}[section]
\newtheorem{corollary}{Corollary}[section]
\newtheorem{proposition}{Proposition}[section]

\newtheorem{remark}{Remark}[section]

\newcounter{theor}

\newtheorem{thm}[theor]{Theorem}

\def\conv{\mathop\mathrm{conv}\nolimits}

\def\s{\mathbb{S}}

\def\E{\mathbb{E}}

\def\R{\mathbb{R}}
\def\N{\mathbb{N}}

\newcommand{\dlat}{\mathrm{d}}

\def \p{\Pi}
\def \pp{\Pi^{\circ}}

\newcommand{\hid}{m}

\def\esc#1{\left\langle #1\right\rangle}

\numberwithin{equation}{section}

\begin{document}
\title[On general versions of the Petty projection inequality]{On general versions of the Petty projection inequality}

\author[F. Marín Sola]{Francisco Marín Sola}
\email{francisco.marin7@um.es}
\address{Departamento de Ciencias, Centro Universitario de la Defensa (CUD), 30729 San Javier (Murcia), Spain. }


\thanks{The author is supported by the grant PID2021-124157NB-I00, funded by MCIN/AEI/10.13039/501100011033/``ERDF A way of making Europe'', as well as by the grant  ``Proyecto financiado por la CARM a través de la convocatoria de Ayudas a proyectos para el desarrollo de investigación científica y técnica por grupos competitivos, incluida en el Programa Regional de Fomento de la Investigación Científica y Técnica (Plan de Actuación 2022) de la Fundación Séneca-Agencia de Ciencia y Tecnología de la Región de Murcia, REF. 21899/PI/22''.}

\begin{abstract}
The classical Petty projection inequality is an affine isoperimetric inequality which constitutes a cornerstone in the affine geometry of convex bodies. By extending the polar projection body to an inter-dimensional operator, Petty's inequality was generalized in \cite{HLPRY23_2} to the so-called $(L_p,Q)$ setting, where $Q$ is an $m$-dimensional compact convex set. In this work, we further extend the $(L_p,Q)$ Petty projection inequality to the broader realm of rotationally invariant measures with concavity properties, namely, those with $\gamma$-concave density (for $\gamma\geq-1/nm$). Moreover, when $p=1$, and motivated by a contemporary empirical reinterpretation of Petty’s result \cite{PPT}, we explore empirical analogues of this inequality. 
\end{abstract}

\maketitle

\section{Introduction}
The classical isoperimetric inequality asserts that among all convex bodies (compact convex sets with non-empty interior) of a given volume, Euclidean balls have the smallest surface area. More precisely, if $K\subset\R^n$ is a convex body and $B_2^n$ is the Euclidean unit ball in $\R^n$,
\begin{equation}\label{e:classical_isoperimetric}
    S(K) \geq S(B_K),
\end{equation}
where $S(\cdot)$ denotes the surface area and $B_K$ is an Euclidean ball with the same volume as $K$. 

A well-known affine strengthening of \eqref{e:classical_isoperimetric} is provided by the Petty projection inequality (see \cite{CMP71}). This result shows that among convex bodies with prescribed volume, the polar projection body $\pp(K)$ (see Section \ref{s:background} for a precise definition) has maximal volume when $K$ is an ellipsoid. In a more detailed manner, for every convex body $K\subset \R^n$
\begin{equation}\label{e:classical_petty}
    |\pp(K)|\leq |\pp(B_K)|,
\end{equation}
where $|\cdot|$ denotes the $n$-dimensional Lebesgue measure or \textit{volume}. The projection body operator and its polar  has attracted a lot of attention during decades becoming a cornerstone of the affine geometry of convex bodies and the Brunn-Minkowski theory. For context and background we refer the interested reader to the excellent books by Gardner \cite{gardner_book}, Gruber \cite{G07} and Schneider \cite{Sch2}.  

Listing all the works related to Petty's inequality \eqref{e:classical_petty} and the projection body  would be out of the scope of this manuscript. Nevertheless, we would  like to point out that the projection body operator was extended to the $L_p$ setting, for $p>1$, by Lutwak, Yang and Zhang in the seminal paper \cite{LYZ00}, where they continued the work of Petty extending \eqref{e:classical_petty} to this realm. Other remarkable extension of this operator is the asymmetric $L_p$ projection body introduced by Ludwig in \cite{ML05}, where she gave a complete classification of it from the perspective of Minkowski valuations (see also \cite{ML02} for extensive study of the classical case). An analogue of \eqref{e:classical_petty} for the latter bodies was proved, among other related results, by Haberl and Schuster in \cite{HS09}. More recently, an empirical version of the classical Petty projection inequality was obtained by Paouris, Pivovarov, and Tatarko \cite{PPT}; their work follows the empirical interpretation of isoperimetric inequalities initiated in \cite{PP12} and further developed in \cite{PP13-2,CEFPP15,PP17-2}.

In this work,  we will be mainly concern about the $(L_p,Q)$ polar projection body, $\pp_{Q,p}(\cdot)$ (see \eqref{e:minkowski_functional_(L_p,Q)}). This generalization was  recently introduced and extensively studied by Haddad, Langharst, Putterman, Roysdon and Ye in \cite{HLPRY23_2}, building on a previous work \cite{HLPRY23}. Among other related results, the following $(L_p,Q)$ Petty projection inequality was established:
\begin{thm}[\cite{HLPRY23_2}]\label{t:Lp-Q_Petty}
    Let $m,n \in \N$, $K \subset\R^n$ and $Q \subset \R^m$ be convex bodies containing the origin in its interior. Then
    $$
    |\pp_{Q,p}( K)|_{_{nm}} \leq  |\pp_{Q,p} (B_K)|_{_{nm}}.
    $$
    Equality holds for $p = 1 $ if and only if $K$ is an ellipsoid; if $p>1$ and $n\geq3$, then there is equality if and only if $K$ is an origin symmetric ellipsoid. When $n=2$ the equality conditions were shown in \cite{Ye24}.
\end{thm}
We would like to remark that the $(L_p,Q)$ polar projection body provides an unifying framework that encompasses the aforementioned generalizations \cite{LYZ00,ML05,HS09}.

Our contributions are twofold. First, following the approach of \cite[Section~8.2]{MilYe} along with a new observation from the $L_p$ setting, we extend Theorem~\ref{t:Lp-Q_Petty} to absolutely continuous rotationally invariant, convex measures (i.e., measures with a $\gamma$-concave density for $\gamma \geq -1/nm$). The result reads as follows.
\begin{theorem}\label{t:Q-Petty}
    Let $K\subset \R^n$ and $Q\subset\R^m$ be convex bodies containing the origin in their interior. Then, for any absolutely continuous rotationally invariant, convex measure $\nu$ in $\R^{nm}$  
    $$
      \nu\left(\pp_{Q,p}(K)\right)\leq  \nu\left(\pp_{Q,p}(B_K)\right).
     $$
\end{theorem}

Secondly, motivated by the aforementioned empirical interpretation of the Petty projection inequality, we shall study empirical versions of the previous result. To briefly introduce the setting  (note that we follow the notation of \cite{PP12}), for vectors $x_1,\dots,x_N \in \mathbb{R}^n$ (with $N\geq 1$) and a convex body $C \subset \mathbb{R}^N$, we denote by $[x_1,\dots,x_N]$ the $n\times N$ matrix with columns $x_i$, so that
$$
[x_1,...,x_N]C = \left\{c^1 x_1 + \cdot \cdot \cdot + c^N x_N : c = (c^1,...,c^N)\in C\right\}.
$$
Accordingly, if $X_1,\dots,X_N$ are random vectors in $\mathbb{R}^n$, then $[X_1,\dots,X_N]C$ is a random set. The empirical version of \eqref{e:classical_petty} obtained in \cite{PPT} takes the following form.
\begin{thm}[\cite{PPT}]\label{t:empirical-Petty}
    Let $C\in\R^N$ be a convex body. If $\{X_i\}_{i=1}^N$ are independent random vectors, respectively distributed according to the densities $\{f_i\}_{i=1}^N$, then, for 
    any absolutely continuous rotationally invariant, log-concave measure $\nu$ in $\R^{n}$ 
    $$
    \E\left[\nu\Bigl(\pp\bigl([X_1,...,X_N]C\bigr)\Bigl)\right]\leq  \E\left[\nu\Bigl(\pp\bigl([X^*_1,...,X^*_N]C\bigr)\Bigr)\right],
    $$
    where $\{X_i^*\}_{i=1}^N$are independent random vectors, respectively distributed according to the symmetric decreasing rearrangement of $\{f_i\}_{i=1}^N$.
\end{thm}

In the spirit of \cite{PPT} and building on the ideas of \cite[Section~8.2]{MilYe}, we prove the following when $p=1$. Note that $\pp_Q$ stands for $\pp_{Q,1}$.
\begin{theorem}\label{t:empirical-Q-Petty}
    Let $C\in\R^N$ be a convex body and $Q\in\R^m$ be a convex body containing the origin in its interior, where $m,N\in\N$. If $\{X_i\}_{i=1}^N$ are independent random vectors, respectively distributed according to the densities $\{f_i\}_{i=1}^N$, then, for 
    any rotationally invariant, convex measure $\nu$ in $\R^{nm}$ 
    $$
    \E\left[\nu\Bigl(\pp_{Q}\bigl([X_1,...,X_N]C\bigr)\Bigr)\right]\leq  \E\left[\nu\Bigl(\pp_{Q}\bigl([X^*_1,...,X^*_N]C\bigr)\Bigr)\right],
    $$
    where $\{X_i^*\}_{i=1}^N$ are independent random vectors, respectively distributed according to the symmetric decreasing rearrangement of $\{f_i\}_{i=1}^N$.
\end{theorem}

Finally, we would like to remark that, as a byproduct of the tools developed to prove Theorem \ref{t:Q-Petty}, we recover a recent result by  Cao, Wang and Wang \cite{CWW} concerning  the $L_p$-analogue of the surface area measure (see Corollary \ref{c:CaoWangWang}).

The paper is organized as follows. In Section \ref{s:background} we collect some preliminaries, background material and several tools the we shall use later on. Section \ref{s:L_p} is devoted to the proof of Theorem \ref{t:Q-Petty} and some remarks related with $L_p$ mixed volumes. Finally, the proof of Theorem \ref{t:empirical-Q-Petty} is provided in Section \ref{s:empirical} 

\section{Background material}\label{s:background}
We shall work with the $n$-dimensional Euclidean space $\R^n$ endowed with their standard inner product $\esc{\cdot,\cdot}$, and  $x^{i}$ are used for the $i$-th coordinate of a vector in such  space. Given any set $M \subset \R^n$, we  use $\chi_{_{M}}$ to denote its characteristic function. For any given unit direction $u\in\s^{n-1}$, we shall use  $u^\perp$, $P_{u^\perp}$ and $R^u$ for a hyperplane with normal vector $u$, the orthogonal projection onto it, and the reflection map about it. Throughout the paper, the $k$-dimensional Lebesgue measure of a measurable set is denoted by $|\cdot|_k$ and we will omit the index $k$ when it is equal to the dimension of the ambient space; furthermore, as usual,  integrating $\dlat x$ will stand for integration with respect to the Lebesgue measure. 

Let $K_1,...,K_m\subset \R^n$ be convex bodies and $\lambda_1,...,\lambda_m \geq 0$. The volume of their Minkowski sum is given by
$$
|\lambda_1K_1 + \cdot \cdot \cdot+\lambda_mK_m| = \sum_{1\leq i_1,....,i_n\leq m}\lambda_{i_1}\cdot\cdot\cdot \lambda_{i_n}V(K_{i_1},...,K_{i_n}),
$$
where the coefficient $V(K_{i_1},...,K_{i_n})$ is the so-called mixed volume of $n$-tuple $(K_{i_1},...,K_{i_n})$ (see \cite[Section~5]{Sch2} for background and properties). We shall use the standard abbreviation $V(K[n-k],L[k])$ when $K$ and $L$ appear, respectively, $n-k$ and $k$ times in the mixed volume with $k\in\{1,...,n-1\}$.

Let $K\subset\R^n$ be a convex body, its support function $h_K$ and polar body $K^\circ$ are given by
$$
h_K(x) = \max_{y\in K}\esc{x,y} \quad \mathrm{and}\quad K^\circ=\{x \in \R^n : h_K(x)\leq 1\}.
$$
If $K$ contains the origin in its interior, its Minkowski functional is defined as $\|x\|_{K} = \inf \{\lambda > 0 : x \in \lambda K\}$. Note that, in this case, $K = \{x \in \R^n : \|x\|_K \leq 1\}$ and $\|\cdot\|_{K^\circ} = h_K(\cdot)$. We shall also recall some definitions; the projection body, $\p K$, of $K$ is the centrally symmetric convex body whose support function is given by $h_{\p K}(u) = |P_{u^\perp}K|_{n-1} = nV(K[n-1],[0,u])$ where, for any given direction $u \in \s^{n-1}$, $[0,u]$ is a segment joining the origin and $u$. Thus the polar projection body, denoted as $\pp K$, is merely $(\p K)^\circ$. 

\subsection{The $(L_p,Q)$ setting}

Let $K,L\subset \R^n$ be convex bodies containing the origin in their interior. The $p$-sum of $K$ and $L$ (introduced by Firey \cite{Firey62}), $K+_p L$, is defined via its support function as
\begin{equation}\label{e:p-sum}
  h_{K+_pL}^p(u) = h^p_K(u) + h^p_L(u),  
\end{equation}
with $u  \in \s^{n-1}$ (see also \cite{LYZ12-Lp} for a more general pointwise definition). The $L_p$ mixed volume, introduced by Lutwak \cite{LE93,LE96}, is  defined as
\begin{equation}\label{e:L_p-mixed_volume}
V_p(K,L) =\frac{p}{n} \lim_{\varepsilon\to 0^+}\frac{|K+_p\varepsilon\cdot L| - |K|}{\varepsilon}.
\end{equation}

In order to introduce the $(L_p,Q)$ polar projection body, it shall be convenient to identify $\R^n$ with the collection of $n\times 1$ column vectors, $\R^m$ with the collection of $1\times m$ row vectors, and $\R^{nm}$ with the collection of $n\times m$ matrices. Of course, this considerations are not needed when the matrix multiplication is not involved.  Fix $k,n,m\in\N$. Given $A\subset \R^{nk}, B\subset \R^{km}$, we have that $A^t=\{x^t:x\in A\}$ and $A.B=\{x.y:x\in A,y\in B\}\subset \R^{nm}$. Here, $x.y$ means the left-to-right matrix multiplication of $x$ and $y$, and $x^t$ means the transpose of the matrix $x$. Given  convex bodies $Q$ and $K$ in $\R^m$ and $\R^n$, respectively, both containing the origin in their interior. The $(L_p,Q)$ polar projection body, introduced in \cite{HLPRY23_2} for $p\geq 1$, is defined via its Minkowski functional as 
\begin{equation}\label{e:minkowski_functional_(L_p,Q)}
\| x\|_{\Pi^{\circ}_{Q,p} K}^p=nV_{p}(K, x.Q^t),
\end{equation}
where $x$ is a non-zero matrix which we identify as a vector in $\R^{nm}$. Note that the classical polar projection body is recovered by taking $p=1$ and $Q=[0,1]$. 

We would also like to remark that \eqref{e:minkowski_functional_(L_p,Q)} encapsulates several generalizations of the polar projection body operator. For example, when $Q = -\Delta_m$, where $\Delta_m$ is the $m$-dimensional orthogonal simplex, the recently introduced $m$th-order case arises \cite{HLPRY23}. If $Q = [-1,1]$ and $p \geq 1$, the $L_p$ polar projection body, originally defined by Lutwak, Yang, and Zhang \cite{LYZ00}, is recovered. Furthermore, for $p > 1$, Ludwig's \emph{asymmetric} $L_p$ polar projection bodies \cite{ML05} correspond to the cases $Q = [0,1]$ and $Q = [-1,0]$.
\subsection{Linear parameter systems}

We shall also recall the notion of a \emph{linear parameter system} introduced by Rogers and Shepard in \cite{RS58:2}  (see also \cite[Section~10.4]{Sch2}). Let $K \subset \R^n$ be a convex body, and $u\in\s^{n-1}$ be a unit direction. A linear parameter system is a family of convex bodies $\{K(t)\}_{t\in I}$, where $I\subset\R$ is an interval, that can be represented as
\begin{equation}\label{e:linear_parameter_system}
    K(t) = \conv\{x_j + \lambda_j t u : j \in \mathcal{J}\}.
\end{equation}
Here $\{x_j\}_{j\in\mathcal{J}}$ and $\{\lambda_j\}_{j\in \mathcal{J}}$ are bounded sets in $\R^n$ and $\R$ respectively, and $\mathcal{J}$ is an arbitrary index set. We shall also make use of the case in which the index set is a convex body $K$, i.e,
$$
K(t) = \text{conv}\{ x + t \alpha(x)u : x\in K\},
$$
where $\alpha : K\to\R$ is a bounded function. A fundamental property of this construction is its convexity under mixed volumes: an elegant proof can be found in \cite[Theorem~10.4.1]{Sch2}.
\begin{theorem}\label{t:convexity_Mixed_Volumes)}
    Let $\{K_i(t)\}_{t\in I}$, with $i = 1,...,n$, be linear parameter systems in the direction $u\in \s^{n-1}$. Then $t \mapsto V\bigl(K_1(t),...,\,K_n(t)\bigr)$ is convex.
\end{theorem}
As a corollary one obtains a classical result proved by Rogers and Shepard \cite{RS58:1}.
\begin{corollary}\label{c:shepard_convexity}
    Let $\{K(t)\}_{t\in I }$ be a linear parameter system. Then $t\mapsto|K(t)|$ is convex.
\end{corollary}

Linear parameter system have demonstrated to be a powerful tool to prove isoperimetric type inequalities (and its reverse counterpart in the plane). To list some fundamental contributions in this regard, we would like to mention the work of Campi and Gronchi \cite{CampiGronchiLpBPC,CG06,CG06_2} and Campi, Colesanti and Gronchi \cite{CCG99}.

In this work, we shall mostly  use a specific linear parameter system. Let $K\subset \R^n$ be a convex body and $u \in \s^{n-1}$. Then $\{K_u(t)\}_{t\in [-1,1]}$ is the linear parameter system given by 
\begin{equation}\label{e:shadow_system}
    K_u(t)= \{y + su: y\in P_{u^\perp}K, \; s\in [f_u^t(y),-g_u^t(y)]\},
\end{equation}
with
$$
f_u^t(y) = \frac{(1+t)f_u(y) + (1-t)g_u(y)}{2}
$$
and 
$$
 g_u^t(y) = \frac{(1-t)f_u(y) + (1+t)g_u(y)}{2},
$$
where $f_u,g_u : P_{u^\perp}K \to \R$ are convex functions such that
$$
K=\{ y + su: y\in P_{u^\perp}K, \; s\in [f_u(y),-g_u(y)]\}.
$$
Note that $\{K_u(t)\}_{t\in [-1,1]}$ interpolates continuously between $K_u(1) = K$, $K_u(-1) = R^u (K)$ and $K_u(0) = S^u(K)$, where $S^u(K)$ is the \emph{Steiner} symmetral of $K$ in the direction $u$. Moreover, $K_u(-t) = R^u\bigl(K_u(t)\bigr)$ and $|K_u(t)| = |K|$ for all $t\in[-1,1]$. In addition, for every $y\in u^\perp$ and $t\in[-1,1]$
$$
[K_u(t)]_y = \frac{1+t}{2}[K]_y + \frac{1-t}{2}[R^u(K)]_y,
$$
where we use the notation $[K]_y$ for the  one-dimensional fiber $K \cap (y + \R)$. The interested reader may check \cite[Lemma~2.1]{MilYe} for further details.

A natural question is whether linear parameter systems are ``preserved'' under the Minkowski addition or, more generally, for the $p$-sum \eqref{e:p-sum} with $p\geq 1$. The following result by Bianchini and Colesanti provides an answer.
\begin{theorem}[\cite{BiaCol08}]\label{t:Bianchini_Colesanti_Lp-sum_shadow_systems}
    Let $\{K(t)\}_{t\in I}$ and $\{L(t)\}_{t\in I}$ be linear parameter systems along the direction $u\in \s^{n-1}$. Then, for every $p\geq 1$ $\{K(t) +_p L(t)\}_{t \in I}$ is also a linear parameter system along the direction $u \in \s^{n-1}$.
\end{theorem}

\subsection{Functional results}

To finish the background section, we recall some functional results and definitions that will be used later on. We say that a function $\varphi:\R^n\longrightarrow\R_{\geq0}$ is
$p$-concave, for $p\in\R\cup\{\pm\infty\}$, if
\begin{equation*}\label{e:p-concavecondition}
\varphi\bigl((1-\lambda)x+\lambda y\bigr)\geq
\bigl((1-\lambda)\varphi(x)^p+\lambda \varphi(y)^p\bigr)^{1/p}
\end{equation*}
for all $x,y\in\R^n$ such that $\varphi(x)\varphi(y)>0$ and any $\lambda\in(0,1)$. The cases $p=0$, $p=\infty$ and $p=-\infty$ follow by continuity of the $p$-means, i.e, one obtains the geometric mean, the maximum and the minimum (of $\varphi(x)$ and $\varphi(y)$), respectively. Note that if $p>0$, then $\varphi$ is $p$-concave if and only if $\varphi^p$ is concave on its support $\{x\in\R^n: \varphi(x)>0\}$. In particular, a $0$-concave function is usually called \emph{log-concave} whereas a $(-\infty)$-concave function is referred to as \emph{quasi-concave}. Moreover, Jensen's inequality ensures  that a $q$-concave function is also $p$-concave, whenever $q>p$.

Let $s\in[-\infty,1]$ and $\nu$ be a Borel measure in $\R^n$. Then $\nu$ is $s$-concave if
$$
\nu\bigl((1-\lambda) A + \lambda B\bigr)\geq \bigl((1-\lambda)\nu(A)^s + \lambda \nu(B)^s\bigr)^{1/s}.
$$
When $s = -\infty$ the measure is usually referred to as \emph{convex}. Following Borell's characterization \cite{Borell}, an absolutely continuous measure $\nu$ in $\R^n$ with density $\varphi$ is $s$-concave if and only if $\varphi$ is $p$-concave with $p=s/(1-ns)$ (note that Jensen's inequality implies that convex measures are the largest class among $s$-concave ones). The latter can be deduced from the following result, originally proved in
\cite{Borell} and \cite{BL} (see also \cite{G} for a detailed presentation).
\begin{theorem}[The Borell-Brascamp-Lieb inequality]\label{t:BBL}
Let $\lambda\in(0,1)$. Let $-1/n\leq p\leq\infty$ and let $f,g,h:\R^n\longrightarrow\R_{\geq0}$ be measurable functions, with positive integrals, such that
\begin{equation*}
h\bigl((1-\lambda)x + \lambda y\bigr)\geq \bigl((1-\lambda)f(x)^p + \lambda g(y)^p\bigr)^{1/p}
\end{equation*}
for all $x,y\in\R^n$ such that $f(x)g(y)>0$. Then
\begin{equation}\label{e:BBL}
\int_{\R^n}h(x)\,\dlat x\geq \left((1-\lambda)\left(\int_{\R^n}f(x)\,\dlat x\right)^q+\lambda\left(\int_{\R^n}g(x)\,\dlat x\right)^{q}\right)^{1/q},
\end{equation}
where $q=p/(np+1)$.
\end{theorem}

Another consequence of the Borell-Brascamp-Lieb inequality is the following. 
\begin{corollary}\label{c:marginals}
    Let $f:\R^n\times\R^d$ be a $p$-concave function, with $p\geq -1/n$. Then
    $$
    F(y) = \int_{\R^n}f(x,y)\,\dlat x
    $$
    is $p/(np+1)$-concave.
\end{corollary}

We conclude by recalling the notion of symmetric decreasing rearrangement.  We essentially follow \cite{Burchard} (see also \cite[Chapter~3]{LiebAnalysis}). Let $A\subset\R^n$ be a measurable set with finite volume. Its symmetric rearrangement $A^*$ is an Euclidean open ball with the same volume as $A$. Let now $f:\R^n \to \R_+$ be an integrable function. Using its \emph{layer-cake representation} 
$$
f(x) = \int_0^{+\infty} \chi_{_{\{f(x)>t\}}}(x)\,\dlat t,
$$
the \emph{symmetric decreasing rearrangement} of $f$, denoted as $f^*$, is given by
$$
f^*(x) = \int_0^{+\infty} \chi_{_{\{f(x)>t\}^*}}(x)\,\dlat t
$$
Note that $f^*$ is radially symmetric and decreasing. Moreover, it preserves the volume of the superlevel sets, i.e, $|\{f(x) > t\}| = |\{f^*(x)>t\}|$ for all $t>0$.

For any given $u\in \s^{n-1}$, the \emph{Steiner} symmetral of $f$ with respect to $u$, denoted as $f^u$ is the function given by 
$$
f^u(x) = \int_0^{+\infty} \chi_{_{\bigl\{S^u(\{f(x)>t\})\bigr\}}}(x)\,\dlat t.
$$
Equivalently, $f^u$ is obtained rearranging $f$ along every line parallel to $u$, i.e., for every $y\in u^\perp$, taking $h(t) = f(y + tu)$, we have that $f^u(y+ tu) = h^*(t)$. It is proved in \cite{BLL74} (see also \cite[Chapter~14]{BSConv11}) that for every measurable function $f:\R^n \to \R_+$ with compact support there exists a sequence of the form $f_0 = f$ and $f_{n+1} = f_{n}^u$ for some $u \in \s^{n-1}$, which converges in the $L_1(\R^n)$-norm to $f^*$.

A result involving symmetric decreasing rearrangements that will be central in this note is Christ's version \cite{Christ84} of Rogers-Brascamp-Lieb-Luttinger's inequality. As shown in \cite{PP12} (see also \cite{PP17-1,PPT}), this theorem is a powerful tool for proving empirical type isoperimetric inequalities. We state it for reader's convenience. 
\begin{theorem}[\cite{Christ84}]\label{t:Christ_BL}
    Let $f_1,...,f_N : \R^n \to  \R_+$ be integrable functions and let $F : ( \R^n)^N \to \R_+$. Suppose that $F$ satisfies that, for any $u \in \s^{n-1}$ and $ y =(y_1,...,y_N) \in (u^\perp)^N$, the function $F_{u,y}:\R^N \to \R_+$ defined by $F_{u,y}(t_1,...,t_N) = F(y_1 + t_1 u, ...,y_N+t_N u)$ is even and quasi-concave. Then 
    \begin{equation}\label{e:Christ_BL}
         \int_{(\R^n)^N}F(x_1,...,x_N)\prod_{i=1}^N f_i(x_i) \, \dlat  x \leq \int_{(\R^n)^N}F(x_1,...,x_N)\prod_{i=1}^N f^*_i(x_i) \, \dlat  x,
    \end{equation}
     where $\dlat  x$ stands for $\dlat x_1\cdots\dlat x_N$. Indeed, under the assumptions of Theorem \ref{t:Christ_BL} one can check that
\begin{align}\label{e:BBL_Steiner}
    \int_{(\R^n)^N}F(x_1,...,x_N)\prod_{i=1}^N f_i(x_i) \, \dlat x \leq \int_{(\R^n)^N}F(x_1,...,x_N)\prod_{i=1}^N f_i^u(x_i) \, \dlat x,
\end{align}
 which implies \eqref{e:Christ_BL} after a sequence of Steiner symmetrizations with respect to suitable directions. We refer the reader to \cite[Proposition~3.2]{PP12} for a detailed exposition of the latter.
\end{theorem}

\section{On the $(L_p,Q)$ Petty's projection inequality}\label{s:L_p}

We start this section by proving a simple observation regarding the interplay between $L_p$-mixed volumes \eqref{e:L_p-mixed_volume} and linear parameter systems.
\begin{lemma}\label{p:Lp-mixed_volumes_convexity}
     Let $\{K(t)\}_{t\in I}$ and $\{L(t)\}_{t\in I}$ be linear parameter systems, both along the direction $u\in \s^{n-1}$ and containing the origin for all $t\in I$, such that $|K(t)|$ does not depend on $t$. Then $t \mapsto V_p\bigl(K(t),L(t)\bigr)$ is convex for every $p\geq 1$.
\end{lemma}
\begin{proof}
    Let $\varepsilon> 0$ and $\phi_\varepsilon: I \to \R$ be the function given by
    $$
    \phi_\varepsilon(t) = \frac{|K(t)+_p\varepsilon\cdot L(t)| - |K|}{\varepsilon}.
    $$
    Theorem \ref{t:Bianchini_Colesanti_Lp-sum_shadow_systems} together with Corollary \ref{c:shepard_convexity} imply that $\phi_\varepsilon$ is a convex function for every $\varepsilon > 0$. Moreover, using that  $K(t)$ and $L(t)$ contain the origin for any $t\in I$, we have that $K(t)+_p\,\varepsilon\cdot L(t)$ is a convex body for every $t \in I$ and $\varepsilon >0$. Hence, the limit $\lim_{\varepsilon \to 0^+}\phi_\varepsilon(t)$ always exists. Finally, taking into account that $V_p\bigl(K(t),L(t)\bigr) = \lim_{\varepsilon \to 0^+}\frac{p}{n}\phi_{\varepsilon}(t)$, the assertion immediately follows.
\end{proof}
Recently, Cao, Wang and Wang \cite{CWW} establish, among other related results, an $L_p$ analogue of the so-called Steiner inequality. Specifically, they proved that, for any convex body $K$ containing the origin in its interior, its $L_p$ surface area $S_p(K)$ does not increase under Steiner symmetrization. Note that $S_p(K) = V_p(K,B_2^n)$, and hence Lemma \ref{p:Lp-mixed_volumes_convexity} implies the following.
\begin{proposition}\label{p:Lp-surface_area_convexity}
    Let $\{K(t)\}_{t\in I}$ be a linear parameter system containing the origin for all $t\in I$. Then,  $t \mapsto S_p\bigl(K(t)\bigr)$ is convex for every $p\geq1$.
\end{proposition}

As a consequence, we recover the result of Cao, Wang and Wang.
\begin{corollary}\label{c:CaoWangWang}
    Let $K\subset\R^n$ be a convex body containing the origin in its interior. Then, for every $u\in\s^{n-1}$,
    $$
    S_p\bigl(S^u(K)\bigr)\leq S_p(K).
    $$
\end{corollary}
\begin{proof}
    Considering $\{K_u(t)\}_{t\in [-1,1]}$ defined in \eqref{e:shadow_system} (which contains the origin in its interior for all $t\in [-1,1]$), we get that
$$
S_p(S^u\bigl(K)\bigr) = S_p\bigl(K_u(0)\bigr) \leq \frac{1}{2}S_p(K) + \frac{1}{2}S_p\bigl(R^u(K)\bigr) = S_p(K)
$$
as desired.
\end{proof}

We need to introduce some notation before proving the main result of the section.  Let $x = (x_1, \dots, x_m) \in \mathbb{R}^{nm}$ and $u \in \mathbb{S}^{n-1}$. By considering the orthogonal decomposition in $\mathbb{R}^n$ with respect to $u$, each vector $x_i$ can be written as
\[
x_i = y_i + s_i u
\]
where $y_i \in u^\perp$ and $s_i \in \mathbb{R}$. Thus, setting $y = (y_1, \ldots, y_m) \in (u^\perp)^m$ and $\bar{s} = (s_1, \ldots, s_m) \in \mathbb{R}^m$. The vector $x$ can then be expressed as $x = y + \bar{s} \otimes u$, where $\otimes$ denotes the usual Kronecker product; that is, $\bar{s} \otimes u = (s_1u, \ldots, s_m u)$.
\begin{proof}[Proof of Theorem \ref{t:Q-Petty}]
    Let $\nu$ be a rotationally invariant, convex measure in $\R^{nm}$ with  $\dlat \nu(x) = \phi(x)\,\dlat x$ ($\phi$ is $-1/nm$-concave). We emphasize that $ x = (x_1,\dots,x_m)$, where each $x_i$ is a vector in $\R^n$. To make notation more compact $\dlat x$ stands for $\dlat x_1 \cdot \cdot \cdot \,\dlat x_m$ whereas $\bar s = (s_1,...,s_m)\in \R^m$.
    
    Let $\{K_u(t)\}_{t \in [-1,1]}$ be the linear parameter system defined in \eqref{e:shadow_system}. Then, we obtain from \eqref{e:minkowski_functional_(L_p,Q)} that
    \begin{align*}
        \nu\Bigl(\pp_{Q,p}\bigl(K_u(t)\bigr)\Bigr) &= \int_{(\R^n)^m} \chi_{_{\{nV_{p}(K_u(t),x.Q^t)\leq 1\}}}(x)\,\phi(x)\,\dlat x\\
        &= \int_{(u^\perp)^m} \int_{\R^m}  \chi_{_{\{nV_{p}(K_u(t), (y + \bar s \otimes u).Q^t)\}}}(y + \bar s \otimes u)\,\phi_{y}(\bar s)\, \dlat \bar s\,\dlat y \\
        &= \int_{(u^\perp)^m} F(t,y)\,\dlat y,
    \end{align*}
    where $F(t,y):[-1,1]\times(u^\perp)^m \to \R_+$ is the  function given by
    $$
    F(t,y)= \int_{\R^m} \chi_{_{\{f_p( t, \bar s) \leq 1\}}}(t,\bar s)\phi_{y}(\bar s)\,\dlat \bar s,
    $$ 
     $\phi_{y}(\bar s) = \phi(y + \bar s \otimes u)$ and $ f_p( t, \bar s) = nV_{p}\bigl(K_u(t),(y + \bar s \otimes u).Q^t\bigr)$. On the one hand, we have that 
    \begin{align*}
    f_p(- t, -\bar s) &= nV_p\bigl(K_u(-t),(y -\bar s \otimes u).Q^t\bigr)\\
    &= nV_p\Bigl(R^u\bigl(K_u(t)\bigr),R^u\bigl((y + \bar s \otimes u).Q^t\bigr)\Bigr)\\
    &=f_p( t, \bar s).
\end{align*}
    
    On the other hand, let  $\lambda\in (0,1)$ and $( t,\bar s ),( t^\prime,\bar s^\prime )\in [-1,1]\times \R^m$. Then,
    \begin{align*}
         f_p&\bigl((1-\lambda) t+ \lambda  t^\prime ,(1-\lambda)\bar s+ \lambda \bar s^\prime \bigl)\\
         &=nV_p\Bigr(K_u\bigl((1-\lambda) t+ \lambda  t^\prime\bigr), \bigl(y + ((1-\lambda)\bar s+ \lambda \bar s^\prime)  \otimes u\bigr).Q^t\Bigr).
    \end{align*}
    Note that $\{K_u\bigl((1-\lambda)t+ \lambda t^\prime\bigr)\}_{\lambda\in (0,1)} = \{K_u(\lambda)\}_{\lambda \in (0,1)}$ is also a linear parameter system. Moreover, $K_u(\lambda)$ contains the origin and $|K_u(\lambda)| = |K|$ for all $\lambda \in (0,1)$. In addition, for any $q = (q^1,...,q^m)\in Q$ we have that $(y + \bar s \otimes u).q^t = q^1(y_1 + s_1u) + ...+(y_m + s_mu)q^m$.  Therefore,
\begin{align*}
   \bigl(y &+ ((1-\lambda)\bar s+ \lambda \bar s^\prime)  \otimes u\bigr).Q^t \\
    &= \left\{\sum_{i=1}^m q^i(y_i + s_i u) + \lambda \sum_{i=1}^m q^i (s_i^\prime - s_i)u : q = (q^1,...,q^m)\in Q\right\}\\
    &= \left\{x_q + \mu_q \lambda u : q \in Q\right\},
\end{align*}
i.e., $ \left\{\bigl(y + ((1-\lambda)\bar s+ \lambda \bar s^\prime)  \otimes u\bigr).Q^t\right\}_{\lambda\in(0,1)}$ is a linear parameter system of the form \eqref{e:linear_parameter_system}, generated by the sets $\{x_q\}_{q\in Q}$ and $\{\mu_q\}_{q\in Q}$, which contains the origin for all $\lambda \in (0,1)$. Hence, using Lemma \ref{p:Lp-mixed_volumes_convexity}, we get that $\lambda \mapsto  f_p\bigl((1-\lambda)t+ \lambda  t^\prime ,(1-\lambda)\bar s+ \lambda \bar s^\prime \bigl)$ is convex, which yields the joint convexity of $f_p( t, \bar s)$. Thus, the function $\varphi:[-1,1]\times(u^\perp)^m \times \R^m$ given by $\varphi(t,y, \bar s) =  \chi_{_{\{f_p( t, \bar s) \leq 1\}}}(t,\bar s)\phi_{y}(\bar s)$ is $(-1/nm)$-concave. Now since
$$
 F(t,y)= \int_{\R^m}\varphi(t,y, \bar s)\,\dlat \bar s,
$$
Corollary \ref{c:marginals} implies that $F(t,y)$ is $\alpha$-concave, with $\alpha =\frac{1}{m(1-n)}$.  Moreover, considering $F_{y} (t) = F(t,y)$ for any fixed $y\in (u^\perp)^m$, by a change of variables and the evenness of $f_p( t,\bar s)$ we get that
\begin{align*}
    F_{y} (- t) &= \int_{\R^m} \chi_{_{\{f_p(- t, \bar s) \leq 1\}}}\phi_{y}(\bar s)\,\dlat \bar s\\
    &= \int_{\R^m} \chi_{_{\{f_p( t, -\bar s) \leq 1\}}}\phi_{y}(\bar s)\,\dlat \bar s\\
    &=\int_{\R^m} \chi_{_{\{f_p( t, \bar s) \leq 1\}}}\phi_{y}(-\bar s)\,\dlat \bar s\\
    &= F_{y} ( t),
\end{align*}
where in the last identity we have used the rotational invariance of the measure $\nu$. The latter, together with the fact that $F_{y} (t)$ is $\alpha$-concave, implies that $F_{y} ( t)  \leq F_{y} (0)$ for all $t\in[-1,1]$ and $y \in (u^\perp)^m$. In particular,
$$
 \nu\left(\pp_{Q,p}(K)\right) =  \int_{(u^\perp)^m} F_{y} (1)\,\dlat y \leq  \int_{(u^\perp)^m} F_{y} (0)\,\dlat y =\nu\left(\pp_{Q,p}\bigr(S^u(K)\bigr)\right).
$$
The proof concludes from the following two facts; on the one hand, for any convex body $K$ one can find a sequence of Steiner symmetrizations which converges (w.r.t. the Hausdorff metric) to an Euclidean ball with same volume as $K$. On the other hand, the operator $\pp_{Q,p}$ is continuous w.r.t. the corresponding Hausdorff metric (see \cite[Proposition~3.8]{HLPRY23_2}). 
\end{proof}

\section{Empirical inequalities}\label{s:empirical}
This section is devoted to the proof of Theorem \ref{t:empirical-Q-Petty}. Although we essentially adapt the ideas used in \cite{PPT}, which build on the argument developed in \cite[Proposition~8.4]{MilYe}, we include a sketch of the proof for the sake of completeness. 

\begin{proof}[Proof of Theorem \ref{t:empirical-Q-Petty}]
     Let $\nu$ be a rotationally invariant, convex measure in $\R^{nm}$ with  $\dlat \nu(x) = \phi(x)\,\dlat x$ and let $u \in \s^{n-1}$ be fixed. For $w = (w_1,\dots, w_m)\in(u^\perp)^{\hid}$ and $y =(y_1,\dots y_N)\in (u^\perp)^N$ we define the function $f : \R^{N} \times \R^m \to \R_+$ given by 
$$
f(\bar t,\bar s) = nV\bigl(C_{y}(\bar t)[n-1],(w+ \bar s \otimes u).Q^t\bigr).
$$
Here we use the notation $C_{y}(\bar t) = [y_1 + t_1u,...,y_N + t_N u]C$ , where $\bar t = (t_1,\dots t_N)\in \R^N$ and $\bar s =(s_1,\dots,s_m)\in \R^m$. 

On the one hand, we have that 
    \begin{align*}
      \nu\Bigl(\pp_{Q}\bigl(C_{y}(\bar t)\bigr)\Bigr) &= \int_{(u^\perp)^m} \int_{\R^m}  \chi_{_{\{nV(C_{y}(\bar t)[n-1],(w+ \bar s \otimes u).Q^t)\leq 1\}}}\,\phi_{y}(\bar s) \,\dlat \bar s\,\dlat w\\
      &=\int_{(u^\perp)^m} \int_{\R^m}  \chi_{_{\{f(\bar t,\bar s)\leq 1\}}}\,\phi_{y}(\bar s)\,\dlat \bar s\,\dlat w \\
        &= \int_{(u^\perp)^m} F_{w,y} ( \bar t)\,\dlat w,
    \end{align*}
     where,  fixed $w$ and $y$, $F_{w,y} : \R^{N} \to \R_+$ is the function given by
$$
    F_{w,y} ( \bar t)= \int_{\R^m} \chi_{_{\{f( \bar t, \bar s) \leq 1\}}}\phi_{w}(\bar s)\,\dlat \bar s 
$$ 
and $\phi_{w}(\bar s) = \phi(w+ \bar s \otimes u)$. On the other hand, using Fubini's theorem
    \begin{align*}
         \E\Big[\nu\Bigl(&\pp_{Q}\bigl([X_1,...,X_N]C\bigr)\Bigr)\Big]\\ &= \int_{(u^\perp)^N}\int_{\R^N} \nu\Bigl(\pp_{Q}\bigl(C_{y}(\bar t)\bigr)\Bigr)\prod_{i=1}^N f_i(y_i + t_iu)\,\dlat \bar t \,\dlat y\\
         & = \int_{(u^\perp)^{N}}\int_{(u^\perp)^m}\int_{\R^N} F_{w, y} ( \bar t)\,\prod_{i=1}^N f_i(y_i + t_iu)\,\dlat \bar t \,\dlat w \,\dlat y.
    \end{align*}
     Let now $\lambda \in (0,1)$ and $t,t^\prime \in \R^N$. Then both 
     $$
     \left\{C_{y}\bigl((1-\lambda)t + \lambda t^\prime\bigr)\right\}_{\lambda\in (0,1)} \quad \mathrm{and}\quad \left\{\bigl(w + ((1-\lambda)\bar s + \bar s^\prime)\otimes u\bigr).Q^t\right\}_{\lambda \in (0,1)}
     $$ 
     are linear parameter systems in the direction $u$. Therefore, Theorem \ref{t:convexity_Mixed_Volumes)} implies that $f(\bar t, \bar s)$ is jointly convex. Moreover, since $C_{y}(- \bar t) = R^u\bigl(C_{y}(\bar t)\bigr)$ and $(w-\bar s \otimes u).Q^t = R^u\bigl((w+ \bar s \otimes u).Q^t\bigr)$, it follows that $f(\bar t, \bar s)$ is even. The latter two facts, together with Corollary \ref{c:marginals} and the rotational invariance of the measure $\nu$, imply that, for any fixed $u\in \s^{n-1}$, $F_{w,y} ( \bar t)$ is an $\alpha$-concave (and, in particular, quasi-concave), even function for every $w \in (u^\perp)^m$ and $y\in (u^\perp)^N$. Hence, using the one-dimensional case of \eqref{e:Christ_BL},
     $$
     \int_{\R^N} F_{w, y} ( \bar t)\,\prod_{i=1}^N f_i(y_i + t_iu)\,\dlat \bar t\leq \int_{\R^N} F_{w, y} ( \bar t)\,\prod_{i=1}^N f^u_i(y_i + t_iu)\,\dlat \bar t
     $$
     for all $w \in (u^\perp)^m$ and $y\in (u^\perp)^N$. As a consequence,
        \begin{align*}
         \E\Big[\nu\Bigl(&\pp_{Q}\bigl([X_1,...,X_N]C\bigr)\Bigr)\Big]\\ &= \int_{(u^\perp)^{N}}\int_{(u^\perp)^m}\int_{\R^N} F_{w, y} ( \bar t)\,\prod_{i=1}^N f_i(y_i + t_iu)\,\dlat \bar t \,\dlat w \,\dlat y\\
         &\leq \int_{(u^\perp)^{N}}\int_{(u^\perp)^m}\int_{\R^N} F_{w, y} ( \bar t)\,\prod_{i=1}^N f^u_i(y_i + t_iu)\,\dlat \bar t \,\dlat w \,\dlat y\\
         &= \E\Big[\nu\Bigl(\pp_{Q}\bigl([X^u_1,...,X^u_N]C\bigr)\Bigr)\Big],
    \end{align*}
    where $\{X_i^u\}_{i=1}^N$ are independent random vectors distributed w.r.t the densities $\{f_i^u\}_{i=1}^N$. Finally, after a sequence of Steiner symmetrizations with respect to suitable directions, we get that
    $$
    \E\left[\nu\Bigl(\pp_{Q}\bigl([X_1,...,X_N]C\bigr)\Bigr)\right] \leq \E\left[\nu\Bigl(\pp_{Q}\bigl([X^*_1,...,X^*_N]C\bigr)\Bigr)\right]
    $$ 
    as we wanted to prove.
\end{proof}

\begin{remark}
    Attending to the arguments used in the proofs of the Theorems \ref{t:Q-Petty} and \ref{t:empirical-Q-Petty}, it may be natural to wonder about possible empirical results in the $L_p$ setting. In this regard, we note that the employed strategies heavily rely on the fact that $\left\{C_y\bigl((1-\lambda)t + \lambda t^\prime\bigr)\right\}_{\lambda\in (0,1)}$ is a linear parameter system and the convexity of mixed volumes. In contrast, when dealing with $L_p$ mixed volumes $V_p(K,L)$, we are only able to prove the convexity of the function $t \mapsto V_p\bigl(K(t),L(t)\bigl)$ when $|K(t)|$ does not depend on $t$  (see Lemma \ref{p:Lp-mixed_volumes_convexity}) which is certainly not the case of $\left\{C_{y}\bigl((1-\lambda)t + \lambda t^\prime\bigr)\right\}_{\lambda\in (0,1)}$.
\end{remark}

{\bf Acknowledgments:}  We sincerely thank the anonymous referee whose valuable remarks have greatly enhanced the clarity and quality of our manuscript.
\bibliographystyle{acm}
\bibliography{references_thesis}
\end{document}